\documentclass[12pt]{amsart}
\usepackage{amssymb,latexsym}
\usepackage{amsfonts}
\usepackage{amsmath}
\usepackage[colorlinks,linkcolor=blue,anchorcolor=blue,citecolor=blue]{hyperref}

\newcommand\C{{\mathbb C}}

\newcommand\Q{{\mathbb Q}}

\newcommand\Z{{\mathbb Z}}

\newcommand\N{{\mathbb N}}

\newcommand\h{\mathrm{h}}
\newcommand\al{\alpha}
\newcommand\ga{\gamma}
\newcommand\be{\beta}

\newtheorem{theorem}{Theorem}[section]
\newtheorem{lemma}[theorem]{Lemma}

\newtheorem{corollary}[theorem]{Corollary}

\theoremstyle{remark}
\newtheorem{remark}[theorem]{Remark}

\numberwithin{equation}{section}

\def\mand{\qquad \text{and} \qquad}

\def\({\left(}
\def\){\right)}

\def\rf#1{\left\lceil#1\right\rceil}

\def\cL{{\mathcal L}}

\begin{document}

\title[Explicit form of Cassels' $p$-adic embedding theorem]
{Explicit form of Cassels' $p$-adic embedding theorem for number fields}

\author{Art\= uras Dubickas}
\address{Department of Mathematics and Informatics, Vilnius University, Naugarduko 24,
LT-03225 Vilnius, Lithuania}
\email{arturas.dubickas@mif.vu.lt}


\author{Min Sha}
\address{School of Mathematics and Statistics, University of New South Wales,
 Sydney NSW 2052, Australia}
%\curraddr{Department of Mathematics and Statistics,
%Case Western Reserve University, Cleveland, Ohio 43403}
\email{shamin2010@gmail.com}

\author{Igor E. Shparlinski}
\address{School of Mathematics and Statistics, University of New South Wales,
 Sydney NSW 2052, Australia}
%\curraddr{Department of Mathematics and Statistics,
%Case Western Reserve University, Cleveland, Ohio 43403}
\email{igor.shparlinski@unsw.edu.au}

%    General info

%\date{Januar1 1, 2001 and, in revised form, June 22, 2001.}



\begin{abstract}
In this paper, we mainly give a general explicit form of Cassels' $p$-adic embedding theorem for number fields. We also give its refined form in the case of cyclotomic fields. As a byproduct, given an irreducible polynomial $f$ over $\Z$, we give a general unconditional upper bound for the smallest prime number $p$ such that $f$ has a simple root modulo $p$.
\end{abstract}

\maketitle

%% The correct journal style for \specialsection is all uppercase; a known bug
%% in amsart.cls prevents this, so input must be uppercase until it is fixed.
%\specialsection*{This is a Special Section Head}

%%%%%%%%%%%%%%%%%%%%%%%%%%%%%%%%%%%%%%%%%%%%%%%%%%%%%%%%%%%%%%%%%%%%%%%%
%\footnote{Here is an example of a footnote. Notice that this footnote
%text is running on so that it can stand as an example of how a footnote
%with separate paragraphs should be written.
%\par
%And here is the beginning of the second paragraph.}%
%%%%%%%%%%%%%%%%%%%%%%%%%%%%%%%%%%%%%%%%%%%%%%%%%%%%%%%%%%%%%%%%%%%%%%%%

\section{Introduction}\label{introduction}

%%IS added subsections
\subsection{Motivation}
We start with recalling a result of  Cassels~\cite{Cassels1976}
that gives a $p$-adic embedding for finitely generated fields of characteristic 0,
which we reproduce here for the convenience of the reader:
\begin{theorem}\label{Cassels}
Let $K$ be a finitely generated extension of the rational field $\Q$, and let $S$ be a finite set of non-zero elements of $K$. Then, there exist infinitely many primes $p$ such that there is an embedding
\begin{equation}\label{embed}
\sigma: K\hookrightarrow \Q_p
\end{equation}
of $K$ into the field of $p$-adic numbers $\Q_p$ for which
$$
|\sigma(\beta)|_p=1, \quad \textrm{for all $\beta\in S$},
$$
where $|~|_p$ denotes the $p$-adic valuation.
\end{theorem}

Theorem~\ref{Cassels} is often a useful tool when one needs to employ $p$-adic techniques to solve various problems in number fields. The point
is that for many natural problems over general fields of characteristic zero, one can expect to get
 a result that is not worse than the corresponding one in the case of an algebraic number field, or even in the case of the field of rational numbers. For example, the above theorem has been used for a long time in the study of recurrence sequences over number fields; see, for example,~\cite{Everest2003,Loxton1977,Poorten1991,Poorten1992,Poorten1995}.

\subsection{Main results}
In this paper, we supplement the methods of Cassels~\cite{Cassels1976}
with several new ingredients and give an explicit version of Theorem~\ref{Cassels} in the case when $K$ is a number field. We believe these new ingredients can be of 
independent interest and may find several other applications. 

To begin with, we state the following general theorem and several subsequent corollaries.
Throughout, for an algebraic number
 $\alpha\in\overline{\Q}$, we denote by $\h(\alpha)$ its (Weil) absolute logarithmic height. For an integer $m\ge 1$, we define $\log^{+}m=\max\{1,\log m\}$, so that $\log^{+}m=\log m$ for all $m \ge 3$.

\begin{theorem}\label{Cassels1}
Let $K$ be a number field of degree $d\ge 2$ generated by $\alpha_1,\ldots,\alpha_m\in K\setminus\Q$ over $\Q$, and let $\beta_1,\ldots,\beta_n$ be some fixed non-zero elements of $K$. Then, there exists a prime number $p$ satisfying
\begin{align*}
p \le m^d\exp&\(d\sum\limits_{i=1}^{m}\h(\alpha_i)\)\\
& \left(dn\sum\limits_{i=1}^{m}\h(\alpha_i)+d\sum\limits_{i=1}^{n}\h(\beta_i)
+dn\log^{+}m\right)^{O(d^2)},
 \end{align*}
 such that~\eqref{embed} holds and
$$
|\sigma(\beta_i)|_p=1,\quad \textrm{for $1\le i\le n$}.
$$
\end{theorem}

\begin{corollary}\label{Cassels11}
Let $K$ be a number field of degree $d\ge 2$ generated by $\alpha_1,\ldots,\alpha_m\in K\setminus\Q$ over $\Q$. Then, there exists a prime number $p$ satisfying
$$
%%IS no need for m^d
p\le %% m^d
\exp\(d\sum\limits_{i=1}^{m}\h(\alpha_i)\)\left(dm\sum\limits_{i=1}^{m}\h(\alpha_i)+dm\right)^{O(d^2)},
$$
 such that~\eqref{embed} holds and
$$
|\sigma(\alpha_i)|_p=1,\quad \textrm{for $1\le i\le m$}.
$$
\end{corollary}

\begin{corollary}\label{Cassels12}
Let $K$ be a number field of degree $d\ge 2$ generated by an algebraic integer $\alpha$ over $\Q$, and let $\beta_1,\ldots,\beta_n\in\Z[\alpha]$ be some fixed non-zero algebraic integers (respectively, units) of $K$. Then, there exists a prime number $p$ satisfying
$$
p\le \exp(d\h(\alpha))\left(d\h(\alpha)+d\right)^{O(d^2)},
$$
 such that~\eqref{embed} holds and
$$
|\sigma(\beta_i)|_p\le1~\textit{(respectively, $|\sigma(\beta_i)|_p= 1$)},\quad \textrm{for $1\le i\le n$}.
$$
\end{corollary}

The above results depend on the generators we choose for $K$ over $\Q$. In contrast, the following bound is independent of the choice of generators, but involves the discriminant of $K$.

\begin{corollary}\label{Cassels13}
Let $K$ be a number field of degree $d\ge 2$ with discriminant $D_K$, and let $\beta_1,\ldots,\beta_n$ be some fixed non-zero elements of $K$. Furthermore, suppose that $K$ has at least one real embedding. Then, there exists a prime number $p$ satisfying
$$
p\le\sqrt{|D_K|} \left(n\log |D_K|+d\sum\limits_{i=1}^{n}\h(\beta_i)\right)^{O(d^2)},
$$
 such that~\eqref{embed} holds and
$$
|\sigma(\beta_i)|_p=1,\quad \textrm{for $1\le i\le n$}.
$$
\end{corollary}

For a prime number $\ell$ and an integer $m$,  we write, as usual, $\ell^e\|m$,
if $e$ is the largest integer with $\ell^e\mid m$.

 Given an integer $m\ge 2$, suppose
that $\ell= P(m)$, where $P(m)$ denotes
the largest prime divisor of $m$ and $\ell^e\|m$. Define
$$
\delta(m)=\left\{ \begin{array}{ll}
                \varphi(m/\ell^e)  &  \textrm{if $\ell\equiv 1 \pmod {m/\ell^{e}}$},\\
                1 & \textrm{otherwise},
                 \end{array} \right.
$$
where $\varphi$ is Euler's totient function. In particular, $\delta(m)=1$ if $m$ is a power of a prime or $m \ge \ell^{e+1}$.

For cyclotomic fields, we can get a refined explicit form of Theorem \ref{Cassels}.
\begin{theorem}\label{Cassels2}
Let $K$ be the $m$-th cyclotomic field with $m> 2$, and let $\beta_1,\ldots,\beta_n$ be some fixed non-zero elements of $K$. Then, there exists a prime number $p$ satisfying
$$
p\le \left(d\sum\limits_{i=1}^{n}\h(\beta_i)+dn\right)^{O(d\delta(m))},
$$
where $d=\varphi(m)$, such that~\eqref{embed} holds and
$$
|\sigma(\beta_i)|_p=1,\quad \textrm{for $1\le i\le n$}.
$$
\end{theorem}

\subsection{Approach}
To prove Theorems~\ref{Cassels1} and~\ref{Cassels2} we, roughly speaking, follow the original proof of Cassels and make each step there explicit. For our purpose, we need to tackle the following three subproblems which appear to be new and  which contain the main techniques in this paper.
%%IS added
We believe that these problems and our contribution to them
can be of independent interest.

Firstly, given generators $\alpha_1,\ldots,\alpha_m$ of $K$ over $\Q$,
we need to construct a primitive element $\alpha$ of $K$ such that $\h(\alpha)$ can be bounded explicitly in terms of heights $\h(\alpha_i)$, $1 \le i \le m$, and $[K:\Q]$. Actually, in Section~\ref{primitive0},
we  study this problem much more than what we need in our particular application.

Secondly, given a primitive element $\alpha$ of $K$ and an arbitrary element $\beta$, $\beta$ can be expressed uniquely as a linear combination of the basis $\{1,\alpha,\ldots,\alpha^{d-1}\}$. We need to bound the heights of the coefficients explicitly. This is handled in Section~\ref{coefficient00}.

Thirdly, given an arbitrary irreducible polynomial $f$ over $\Z$, we need to derive an upper bound for the smallest prime $p$ such that $f$ has a simple root modulo $p$. We study this problem extensively by using elementary arguments in Section~\ref{polynomial0}.

Now, we give a brief outline of the proof of Theorem~\ref{Cassels1}. We first construct a primitive element $\alpha$ of $K$ with bounded height from the given generators $\alpha_1,\ldots,\alpha_m$. Let $f$ be the minimal polynomial of $\alpha$ over $\Z$. Put $\beta_{n+i}=\beta_i^{-1}$ for $1\le i\le n$. Then, for $1\le i\le 2n$, we express $\beta_i$ as a linear combination of the basis $\{1,\alpha,\ldots,\alpha^{d-1}\}$ such that all the coefficients are in reduced form, and denote by $b_i$ the least common multiple of the denominators of the coefficients. Note that a prime $p$ is
%%IS what we need
suitable
 if it satisfies the following two conditions:
\begin{itemize}
\item $f$ has a simple root modulo $p$.

\item $p$ does not divide any $b_i$, $1\le i\le 2n$.
\end{itemize}
Using the results and techniques developed in solving
the above three subproblems, we
%%%IS can expect to
derive an upper bound for the  smallest  such prime $p$;
see Section~\ref{embedding0} for more details.

Throughout the paper, we use the Landau symbols $O$ and $o$. Recall that the assertion $U=O(V)$ is equivalent to the inequality $|U|\le cV$ with some constant $c$, while $U=o(V)$ means that $U/V\to 0$.

\section{``Height'' of a number field}
\label{primitive0}

\subsection{Definitions and main results}

Let $K$ be a number field generated by $\alpha_1, \alpha_2, \ldots, \alpha_m$ over $\Q$. In this section, we show the existence of  a primitive element $\alpha$ of $K$
of small height. We present more general versions than we actually need for our purpose.

Given a polynomial $f(x)=a_dx^d+\cdots+a_0=a_d (x-\al_1)\cdots (x-\al_d) \in \C[x]$, where $a_d \ne 0$, its {\it height} is defined by $H(f)=\max_{0 \leq i \leq d} |a_i|$,
and its {\it Mahler measure} by
$$
M(f)=|a_d| \prod_{i=1}^d \max\{1,|\al_i|\}.
$$
For each $f \in \C[x]$ of degree $d$, these quantities are related by the following inequality
\begin{equation}\label{trecia}
H(f) 2^{-d} \leq M(f) \leq H(f) \sqrt{d+1}.
\end{equation}

The left inequality of~\eqref{trecia} follows from the identity
$$
a_{d-i}=(-1)^{i} a_d \sum_{1\leq j_1<\cdots <j_i \leq d} \al_{j_1}\cdots \al_{j_i},
$$
since each product $|a_d \al_{j_1}\cdots \al_{j_i}|$ does not exceed $M(f)$
(see, for example,~\cite[Lemma~3.11]{wa}). The right inequality of~\eqref{trecia} follows from the so-called Landau's inequality $M(f) \leq \sqrt{\sum_{i=0}^d |a_i|^2}$ which was proved, for instance, in~\cite{ca}, \cite{la} and~\cite{sp}.

For an algebraic number $\al \in \overline{\Q}$ of degree $d$, its Mahler measure $M(\al)$ is the Mahler measure of its minimal polynomial $f$ over $\Z$, that is, $M(\al)=M(f)$. Then, the \emph{(Weil) absolute logarithmic height} $\h(\al)$ of $\al$ is equal to $d^{-1} \log M(\al)$. We also define the usual \emph{height} $H(\al)$ of $\al$ as the height of $f$, namely, $H(\al)=H(f)$.

\begin{theorem}\label{primitive}
Let $\al_1,\ldots,\al_m$ be some algebraic numbers of degree $d_1,\ldots,d_m \geq 2$, respectively, and let
$K=\Q(\al_1,\ldots,\al_m)$ be
of degree $d$ over $\Q$. Then, $K$ contains an algebraic number $\al$
satisfying $K=\Q(\al)$ and such that
$$
\h(\al)\leq  \log (m \lfloor d/2 \rfloor)+\h(\al_1)+\cdots+\h(\al_m).
$$
\end{theorem}

Equivalently, the bound of Theorem~\ref{primitive} can be written as
$$
M(\al) \leq (m \lfloor d/2 \rfloor)^d \prod_{i=1}^m M(\al_i)^{d/d_i}.
$$

\begin{corollary}\label{pag0}
Let $\al_1,\ldots,\al_m$ be some algebraic numbers of degree $d_1,\ldots,d_m \geq 2$ and usual height $H_1,\ldots,H_m$, respectively, and let
$K=\Q(\al_1,\ldots,\al_m)$ be of degree $d$ over $\Q$. Then, $K$ contains an algebraic number $\al$
satisfying $K=\Q(\al)$ and
$$
H(\al) \leq (md)^d \prod_{i=1}^m (d_i+1)^{d/(2d_i)}  \prod_{i=1}^m H_i^{d/d_i}.
$$
\end{corollary}

\begin{corollary}\label{pag}
Let $f \in \Z[x]$ be a polynomial of degree $d$ with height $H$
whose splitting field $K$ is of degree $D$ over $\Q$. Then, for some algebraic
number $\al$ satisfying $K=\Q(\al)$, we have
$$
\h(\al)\le \log((d-1) \lfloor D/2 \rfloor)+\frac{d-1}{d}\log (H\sqrt{d+1}),
$$
and
$$
H(\al) \leq (d-1)^{D} D^D (d+1)^{(d-1)D/(2d)} H^{(d-1)D/d}.
$$
\end{corollary}

\subsection{Preparations}

To prove the above results, we  use the following two known facts.

\begin{lemma}\label{pirma}
Let $K$ be a separable extension of degree $d>1$ of a field $F$. Suppose
$K=F(\al_1,\ldots,\al_m)$. Then, for any finite subset $S$ of $F$, there are at least $|S|^{m-1}(|S|-d+1)$ $m$-tuples $(b_1,\ldots,b_m) \in S^m$ for which the element
$\al=b_1\al_1+\cdots+b_m\al_m$ is primitive for $K$ over $F$, namely,
$K=F(\al)$.
\end{lemma}

\begin{lemma}\label{antra}
Let $f \in \Z[x_1,\ldots,x_m]$ be a non-zero polynomial in $m$ variables.  Then, for any algebraic numbers $\ga_1,\ldots,\ga_m$, we have
$$
\h(f(\ga_1,\ldots,\ga_m)) \leq \log L(f)+\sum_{i=1}^m \h(\ga_i) \deg_{x_i} f,
$$
where $\deg_{x_i} f$ is the partial degree of $f$, and $L(f)$ is the sum of moduli of the coefficients of $f$.
\end{lemma}

Lemma~\ref{pirma} is the main result of~\cite{wea} (see also~\cite[Lemma~3.3]{wid} for a slightly weaker result), whereas Lemma~\ref{antra} is exactly~\cite[Lemma 3.7]{wa}.

\subsection{Proofs}

\begin{proof}[Proof of Theorem~\ref{primitive}] We apply Lemma~\ref{pirma} to
$$F=\Q \mand
S=\{-\lfloor d/2 \rfloor, \ldots, \lfloor d/2 \rfloor\}
$$
(note that $d>1$). Since $|S|=2 \lfloor d/2 \rfloor+1 \geq d$,
the number $|S|^{m-1}(|S|-d+1) \ge |S|^{m-1}$ is positive. Thus,
there are some $m$ (not necessarily distinct) integers $b_1,\ldots,b_m \in S$ such that the element $\al=b_1\al_1+\cdots+b_m\al_m$ satisfies $K=\Q(\al)$. Applying Lemma~\ref{antra} to the polynomial $f(x_1,\ldots,x_m)=b_1 x_1+\cdots+b_m x_m$ of length $L(f) = |b_1|+\cdots+|b_m|
\leq m \lfloor d/2 \rfloor$, with $\ga_1=\al_1, \ldots, \ga_m=\al_m$, we deduce
\begin{align*}
 \frac{\log M(\al)}{d}  =\h(\al)
 &=\h(f(\al_1,\ldots,\al_m)) \\
 &\leq  \log (m \lfloor d/2 \rfloor)+\h(\al_1)+\cdots+\h(\al_m) \\
 & =\log (m \lfloor d/2 \rfloor) +\log \(\prod_{i=1}^m M(\al_i)^{1/d_i}\).
 \end{align*}
This implies the required inequalities of Theorem~\ref{primitive}.
\end{proof}

\begin{proof}[Proof of  Corollary~\ref{pag0}] Observe that,
by the right inequality of~\eqref{trecia},
we have $M(\al_i) \leq H_i \sqrt{d_i+1}$
for $i=1,\ldots,m$. Thus,
$$
\prod_{i=1}^m M(\al_i)^{d/d_i} \leq \prod_{i=1}^m H_i^{d/d_i} \prod_{i=1}^m (d_i+1)^{d/(2d_i)}.
$$
Now, selecting $\al$ as in Theorem~\ref{primitive}, we have $\deg \al=d$. Hence, by the left inequality
of~\eqref{trecia} and Theorem~\ref{primitive}, we obtain
\begin{align*}
H(\al)  & \leq 2^d M(\al) \leq  (md)^d \prod_{i=1}^m M(\al_i)^{d/d_i} \\& \leq  (md)^d \prod_{i=1}^m (d_i+1)^{d/(2d_i)}  \prod_{i=1}^m H_i^{d/d_i},
 \end{align*}
as claimed.
\end{proof}

\begin{proof}[Proof of Corollary~\ref{pag}] We write the polynomial  $f \in \Z[x]$ in the form
$f=f_0 f_1^{n_1} \cdots f_q^{n_q}$, where $f_1,\ldots,f_q \in \Z[x]$ are
distinct irreducible polynomials
of degrees $d_1,\ldots,d_q \geq 2$, respectively, and $f_{0} \in \Z[x]$  is a product of linear polynomials.
Assume that $q \geq 1$, since otherwise the claim is trivial, by taking $\al=1$. Thus, $D>1$. Furthermore, in view of
$$
d=n_1d_1+\ldots+n_qd_q+\deg f_{0},
$$
we have $d_i \leq d$ for each $i=1,\ldots,q$.

Put $m=d_1-1+\cdots+d_q-1$. It is clear that the splitting field $K$ of $f$ is generated by arbitrary $d_1-1$ roots
of $f_1$, arbitrary $d_2-1$ roots
of $f_2$, $\ldots$, arbitrary $d_q-1$ roots of $f_q$.
By Theorem~\ref{primitive}, there is an algebraic number $\al \in K$ satisfying $K=\Q(\al)$ and
$$
M(\al) \leq (m \lfloor D/2 \rfloor)^D \prod_{i=1}^q M(f_i)^{(d_i-1)D/d_i},
$$
since we have $d_i-1$ copies of $M(f_i)$ for each $i=1,\ldots,q$.
Using $(d_i-1)/d_i \leq (d-1)/d$ (which follows from $d_i \leq d$) and
\begin{align*}
M(f_1) \cdots M(f_q) & = M(f_1\cdots f_q)\\
& \leq M(f_1\cdots f_q)M(f_0f_1^{n_1-1}\cdots f_q^{n_q-1})=M(f)
\end{align*}
(which follows from the multiplicativity of the Mahler measure and $M(f_i) \geq 1$),
we find that
$$
M(\al) \leq (m \lfloor D/2 \rfloor)^D M(f)^{(d-1)D/d}.
$$
Note that $m \leq d-1$, and by the right inequality of~\eqref{trecia}, $M(f) \leq H(f) \sqrt{d+1}=H\sqrt{d+1}$. Therefore, using these estimates and applying the left inequality of~\eqref{trecia}, we find that
\begin{align*}
\h(\al)=\frac{\log M(\al)}{D}&\le \log(m \lfloor D/2 \rfloor)+\frac{d-1}{d}\log (M(f))\\
&\le \log((d-1) \lfloor D/2 \rfloor)+\frac{d-1}{d}\log (H\sqrt{d+1}),
\end{align*}
and
\begin{align*}
H(\al)  \leq 2^D M(\al)
&\leq (mD)^D M(f)^{(d-1)D/d}\\
& \leq ((d-1)D)^D (H \sqrt{d+1})^{(d-1)D/d}\\
&= (d-1)^D D^D (d+1)^{(d-1)D/(2d)} H^{(d-1)D/d},
\end{align*}
as claimed.
\end{proof}

\section{Bounding the heights of coefficients}\label{coefficient00}

\subsection{Main result}

Let $L/K$ be a number field extension of degree $d\ge 2$, and $L=K(\alpha)$. Then, for any non-zero $\beta\in L$, there exist some $a_0, a_1,\ldots,a_{d-1}\in K$ such that
$$
\beta=a_0+a_1\alpha+\cdots+a_{d-1}\alpha^{d-1}.
$$
Now, we  bound the height of each coefficient $a_i$, $0\le i\le d-1$,
as follows:

\begin{theorem}\label{coefficient0}
Let $L/K$ be a number field extension of degree $d\geq 2$, and $L=K(\alpha)$. Given non-zero $\beta\in L$, and $a_0, a_1,\ldots,a_{d-1}\in K$, such that
$$
\beta=a_0+a_1\alpha+\cdots+a_{d-1}\alpha^{d-1},
$$
we have
$$
\h(a_i)\leq d\h(\be)+3d(d-1)\h(\al)+ d \log {d-1 \choose i}+d(d-1)\log 2+\log d,
$$
for $i=0,1,\ldots,d-1$.
\end{theorem}

Note that, since for the binomial coefficients we have
$${d-1 \choose i} \leq 2^{d-1},\quad i=0,1,\ldots,d-1,
$$
Theorem~\ref{coefficient0} implies that
$$
\h(a_i)\leq d\h(\be)+3d(d-1)\h(\al)+2d(d-1)\log 2+\log d
$$
for each $i=0,1,\ldots,d-1$. This implies the following corollary.
\begin{corollary}\label{coefficient}
{\it
Under the same assumptions and notation as in Theorem~\ref{coefficient0}, we have
$$
\h(a_i)< d\h(\be)+3d^2\h(\al)+ 2d^2,
$$
for $i=0,1,\ldots,d-1$.
}
\end{corollary}

\subsection{Proof of Theorem~\ref{coefficient0}}

In the sequel, we  use the following formulas without special reference (see, e.g., \cite{wa}). For any $n\in \mathbb{Z}$ and $b_{1},\cdots,b_{k},\gamma\in \overline{\mathbb{Q}}$, we have
\begin{align*}
&\h(b_{1}+\cdots+b_{k})\le \h(b_{1})+\cdots+\h(b_{k})+\log k,\\
&\h(b_{1}\cdots b_{k})\le \h(b_{1})+\cdots+\h(b_{k}),\\
&\h(\gamma^{n})=|n|\h(\gamma),\\
& \h(\zeta)=0 \textrm{\quad for any root of unity $\zeta\in \overline{\Q}$}.
\end{align*}

We now assume that $\al_1=\al, \al_2,\ldots,\al_d$ are the conjugates of $\al$ over the field $K$.
Put
$$
 \be_i=\sum_{j=0}^{d-1} a_j \al_i^j, \qquad \text{for} \ i=1,\ldots,d.
$$
So, $\h(\al_i)=\h(\al)$ and $\h(\be_i)=\h(\be)$ for $1\le i\le d$.

To solve the above system of $d$ linear equations in $d$ unknowns $a_0,\ldots,a_{d-1}$, we denote the appearing Vandermonde matrix by
$V=\left(\al_i^{j-1}\right)_{1 \leq i,j \leq d}$. By~\cite[Formula (6)]{Klinger1967}, the inverse of $V$ is given by
$$
V^{-1}=\left( \frac{(-1)^{i+j} \sigma_{d-j}(\al_1,\ldots,\widehat{\al_{i}},\ldots,\al_d)}{\prod\limits_{m=1}^{i-1}(\al_{i}-\al_m)\prod\limits_{k=i+1}^{d}(\al_{k}-\al_{i})}\right)_{1 \leq i,j \leq d}^{T},
$$
where $T$ stands for the transpose, and $\sigma_{k}(\al_1,\ldots,\widehat{\al_{i}},\ldots,\al_d)$ stands for the $k$-th symmetric function in the $d-1$ variables
$\al_1,\ldots,\al_d$ without $\al_{i}$; for instance, in the case $i=d$, we have
$\sigma_{1}(\al_1,\ldots,\al_{d-1})=\al_1+\cdots+\al_{d-1}$ and $\sigma_{d-1}(\al_1,\ldots,\al_{d-1})=\al_1\cdots \al_{d-1}$.

Hence,
\begin{equation}\label{uyt}
 a_{j-1} = \sum_{i=1}^d \beta_{i} \frac{(-1)^{i+j} \sigma_{d-j}(\al_1,\ldots,\widehat{\al_{i}},\ldots,\al_d)}{\prod\limits_{m=1}^{i-1}(\al_{i}-\al_m)\prod\limits_{k=i+1}^{d}(\al_{k}-\al_{i})}.
\end{equation}

Since $\sigma_{d-j}(\al_1,\ldots,\widehat{\al_{i}},\ldots,\al_d)$
is a polynomial with coefficients $1$ in $d-1$ variables
 $\al_1,\ldots,\al_d$ (without $\al_{i}$) of degree $d-j$, length ${d-1 \choose d-j}$, and degree $1$ in each variable $\al_k$, $k \ne i$, by Lemma~\ref{antra}, we find that
\begin{align*}
\h(\sigma_{d-j}(\al_1,\ldots,\widehat{\al_{i}},\ldots,\al_d))
&\leq \log {d-1 \choose d-j}+\sum_{k \ne i} \h(\al_k)\\
&=\log {d-1 \choose d-j}+(d-1)\h(\al).
\end{align*}
On the other hand, in order to bound the denominator of~\eqref{uyt} we observe that
$$
\h\(\prod_{k \ne i} (\al_k-\al_i)\)\le \sum_{k \ne i} \h(\al_k-\al_i) \leq (2d-2)\h(\al)+(d-1)\log 2,
$$
since each term $\h(\al_k-\al_i)$ does not exceed $2\h(\al)+\log 2$.

Thus, the absolute logarithmic height of each of the $d$ summands in~\eqref{uyt} is bounded from above by
$$
\h(\be)+(3d-3)\h(\al)+ \log {d-1 \choose d-j} + (d-1)\log 2.
$$

Hence,  we conclude that
\begin{align*}
\h(a&_{j-1}) \\ &\leq d \(\h(\be)+(3d-3)\h(\al)+ \log {d-1 \choose d-j} + (d-1)\log 2 \)+\log d
\end{align*}
for $j=1,\ldots,d$.  By replacing $j-1$ by $i$ and observing that
$${d-1 \choose d-j}={d-1 \choose d-i-1} = {d-1 \choose i},
$$
we see that this is exactly the required inequality of Theorem~\ref{coefficient0}.

\section{Simple roots of polynomials modulo a prime}\label{polynomial0}

\subsection{Background and main results}

In this section, given an irreducible polynomial $f\in\Z[X]$, we
derive an upper bound for the smallest prime $p$ such that $f$ has a simple root modulo $p$.

First of all, we mention a sharp upper bound of Bella\"iche~\cite{Bellaiche2013} under assumption that both the Generalized Riemann Hypothesis and the Artin Conjecture are true for the Artin $L$-functions associated to the irreducible representations of $G$,
where $G$ is the Galois group of the splitting field of $f$ over $\Q$. Namely, under the above assumptions,
by Bella\"iche~\cite[Th\'eor\`eme 16]{Bellaiche2013},
if $M$ is the product of all the distinct prime divisors of the discriminant of
 a monic irreducible polynomial $f \in \Z[x]$ of degree $d\ge 1$, then
\begin{itemize}
\item  there exists a prime $p=O(d^2(\log M+d\log d)^2)$ such that $p\nmid M$ and $f$ has at least one root modulo $p$;

\item There exists a prime $p=O(d^4(\log M+d\log d)^2)$ such that $p\nmid M$ and $f$ has at least two roots modulo $p$.

\end{itemize}

Here, we  give unconditional upper bounds of such smallest prime $p$ for any irreducible polynomial $f\in\Z[X]$ without assuming that $f$ is monic.
In fact, for our purpose we need a slightly more general result
where $p$ also avoids divisors of a given integer $Q$.

Assume first that the polynomial $f$ which we consider is of degree $1$. Then, we can take the smallest prime $p$ which is coprime to the leading coefficient of $f$. So in the sequel, we suppose that the degree of $f$ is greater than or equal to $2$.

We first give a generic approach on how to find such a prime $p$, which yields a rather simple upper bound for $p$.

\begin{theorem}\label{generic}
Given an irreducible polynomial $f=a_dX^d+\cdots+a_1X+a_0\in\Z[X]$ of degree $d\ge 2$ and of
height $H$, there exists a prime number
$$
p\le\left\{ \begin{array}{ll}
                H,  & \textrm{if $\gcd(a_0,M)=1$ and $|a_0|>1$},\\
                2H(dM)^d,  & \textrm{if $|a_0|=1$},\\
                2H(dHM)^d, & \textrm{if $\gcd(a_0,M)>1$},
                 \end{array} \right.
$$
where $M$ is the product of all the distinct prime divisors of the discriminant of $f$,
such that $f$ has a simple root modulo $p$.
\end{theorem}

We now present an upper bound for such a prime $p$ which behaves much better
than that of Theorem~\ref{generic} with
respect to $H$ (however, in some cases Theorem~\ref{generic}
is still stronger). In fact, we present it in a slightly more general
form.

\begin{theorem}\label{polynomial}
Given an irreducible polynomial $f\in\Z[X]$ of height $H$ and of degree $d\ge 2$, and an integer $Q \ge 3$.
 Then, there exists a prime number $p$ satisfying
$$
 p\le  C^d H\(d\log Q\log^{+}H \)^d+
H (\log Q)^{c d^2},
$$
where  $c$ and $C$  some absolute constants,
such that $f$ has a root modulo $p$ and $p\nmid Q$.
\end{theorem}

We denote the discriminant of $f$ by $\Delta$. Choosing $Q= 3|\Delta|$ we derive the following:

\begin{corollary}\label{single root}
Given an irreducible polynomial $f\in\Z[X]$ of height $H$ and of degree $d\ge 2$, there exists a prime number $p$ satisfying
$$
p\le H(d\log^{+} H)^{O(d^2)},
$$
such that $f$ has a simple root modulo $p$.
\end{corollary}

\begin{remark}
Let $f$ be the $n$-th cyclotomic polynomial with $n>2$. Then, it is well-known that, for a prime $p$, $f$ has a simple root modulo $p$ if and only if
$p\equiv 1 \pmod n$. Linnik's theorem says that such a prime $p$ can be chosen so that $p = O(n^L)$, where $L$ is an absolute constant. A recent result of Xylouris~\cite{Xylouris2011} says that we can choose $L=5.18$.
\end{remark}

\subsection{Products of polynomial values}

 First, we give a lower bound on the product of polynomial values which
 is necessary for our argument and which can be of independent
 interest.

\begin{lemma}\label{product}
Let $f \in \C[x]$ be a polynomial of degree $d \geq 1$, and assume that the absolute value of the leading coefficient of $f$ is not less than 1. Then for each integer
$L \geq 51(2d+1)$, we have
$$
\prod_{j=1}^L \max\{1,|f(j)|\} \geq (L/5)^{dL/18}.
$$
\end{lemma}

\begin{proof}
Call a point $j \in S=\{1,2,\ldots,L\}$ {\it good} if the distance from $j$ to the nearest root
of $f$ is at least $1$. Then $|f(j)| \ge 1$. Since each open disc of radius $1$ and center at a root of $f$ contains at most two points of the set $S$, there are at least $L-2d$
good points in $S$.

Consider four open discs $D_1,D_2,D_3,D_4$ of radius $L/6$ each, with centers at $L/10$, $11L/30$, $19L/30$, $9L/10$, respectively, and put $D_5:=\C \setminus \bigcup_{j=1}^4 D_j$. It is easy to see that the distance from each point of the set
$S$ to $D_5$ is at least
$$
\min\{L/6-L/10, \sqrt{(L/6)^2-(2L/15)^2}\}=\min\{L/15,L/10\}=L/15.
$$
Now, if at least
$d/10$ roots of $f$ lie in $D_5$, we obtain $|f(j)| \geq (L/15)^{d/10}$
for each good $j \in S$.   Thus, as $L \geq 100d$, we deduce
\begin{align*}
\prod_{j=1}^L \max\{1,|f(j)|\} \ge \prod_{j - \text{good}} |f(j)|
&\geq (L/15)^{(L-2d)d/10} \\
&> (L/15)^{2dL/21}>(L/5)^{dL/17},
\end{align*}
which is stronger than required.

Alternatively, when $D_5$ contains less than $d/10$ roots of $f$, the union
$\bigcup_{j=1}^4 D_j$ must contain at least $0.9d$ roots of $f$. Thus, some $D_i$,
where $i \in \{1,2,3,4\}$, contains at least $0.225d$ roots of $f$.
Now, we put $k=1$ if $i=3$ or $i=4$, and $k=4$ if $i=1$ or $i=2$.
The set $D_k$ contains at least $4L/15-2d-1 \geq 0.247L$ good points of $S$. (Here, we use the bound $L \geq 51(2d+1)$.) The distance between any two points of $D_k$ and $D_i$ is at least $19L/30-L/6-(L/10+L/6)=L/5$.
Consequently, the distance  from each  good point in $D_k$ to $D_i$ is at least $L/5$. Thus,
$$
\prod_{j=1}^L \max\{1,|f(j)|\} \ge \prod_{j - \text{good in} \> D_k} |f(j)|  \geq (L/5)^{0.247L \cdot 0.225d} > (L/5)^{dL/18}.
$$
This completes the proof.
\end{proof}

Note that the lower bound of Lemma~\ref{product} is sharp up to the constants. For instance,
for $f(x)=x^d$, we have
$$
\prod_{j=1}^L \max\{1,|f(j)|\} = L!^d \leq L^{dL}.
$$

\subsection{Polynomial congruences}
\label{eq:poly cong}

For a polynomial $f\in \Z[X]$ of degree $d\ge 1$, and two positive integers $L$ and $q$, we define
$$
N(L,q)=|\{1\le j\le L~:~f(j)\equiv 0  \pmod q\}|,
$$
and $N(q)=N(q,q)$.

Recall that the \emph{content} of a polynomial $f$ is defined as the greatest common divisor of the coefficients of $f$. We also need the following three bounds on $N(L,q)$ when
$q = \ell^k$ is a prime power.

\begin{lemma}\label{Lagrange}
Given a positive integer $k$ and a prime number $\ell$. Suppose that the content of $f$ is coprime to $\ell$, and that $f$ has $m$ distinct zeros over $\C$. Then, we have
$$
N(\ell^{k})\le m\ell^{k-1}.
$$
\end{lemma}

\begin{lemma}\label{Konyagin1}
Given a positive integer $k$ and a prime number $\ell$. Suppose that the content of $f$ is coprime to $\ell$. Then, we have
$$
N(\ell^{k})\le 2\ell^{k(1-1/d)}.
$$
\end{lemma}

\begin{lemma}\label{Konyagin2}
Given positive integers $L,k$, and a prime number $\ell$, we have
$$
|N(L,\ell^{k})-\frac{L}{\ell^{k}}N(\ell^{k})|< d.
$$
\end{lemma}

Lemma~\ref{Lagrange} is well-known and also trivial, Lemmas~\ref{Konyagin1} and~\ref{Konyagin2} follow directly from~\cite[Lemma~2]{Konyagin1979} and~\cite[Theorem~1]{Konyagin1994}, respectively.

\subsection{Prime divisors of polynomial products}
\label{eq:poly prod}

The following uniform lower bound on the number of prime divisor is
one of  our main technical tools but may also be of independent interest. 

As usual, let $\omega(k)$ denote the
number of  distinct prime divisors of an integer $k\ge 1$.

\begin{lemma}\label{lem:div prod}
There are absolute  constants $c_1,c_2>0$
such that for any polynomial $f(X) \in \Z[X]$ of degree $d\ge 1$ and 
of height $H$, for each integer
$L \geq 2d+1$,  for the product
$$
W(L) =\prod_{j=1}^L \max\{1,|f(j)|\} 
$$
we have 
$$
\omega\(W(L)\) \ge \min\left\{\frac{c_1L}{\log^+H}, L^{c_2/d}\right\}. 
$$
\end{lemma}

\begin{proof} Let $t=\omega(W(L))$ be the number of distinct prime divisors of $W(L)$.
Since $L \geq 2d+1$, we obviously have $W(L) \ge 2$, so we also have $t \ge 1$. Thus, 
adjusting the constant $c_1$ we can assume that $L \ge 51(2d+1)$.

For a prime $\ell$, we define $r_\ell(L)$ by
$$
\ell^{r_\ell(L)}\|W(L).
$$
Then, we have
$$
r_\ell(L)=\sum\limits_{k=1}^{K_\ell(L)}N(L,\ell^k),
$$
where $N(L,\ell^k)$ is as in Section~\ref{eq:poly cong} and
$$
K_\ell(L)=\max\{r~:~\exists~ 1\le j\le L,\ \ell^r\mid f(j), f(j)\ne 0\}.
$$
Clearly, $|f(j)|\le 2HL^d$ for $1\le j\le L$. Therefore,
\begin{equation}
\label{eq:bound K}
K_\ell(L)\le \log(2HL^d)/\log \ell.
\end{equation}

We use Lemma~\ref{Lagrange} for $k\le d$ and Lemma~\ref{Konyagin1} for $k>d$. Furthermore, from Lemma~\ref{Konyagin2}, we find that
\begin{align*}
r_\ell(L)&\le L\sum\limits_{k=1}^{K_\ell(L)}\frac{N(\ell^k)}{\ell^k}+dK_\ell(L)\\
&\le L\sum\limits_{k=1}^{d}\frac{d}{\ell}+
L\sum\limits_{k=d+1}^{\infty}2\ell^{-k/d}+dK_\ell(L)\\
&=d^2L\ell^{-1}+\frac{2L\ell^{-1}}{\ell^{1/d}-1}+dK_\ell(L).
\end{align*}
Notice that, since $\log x\le x-1$ for $x>0$, we have
$$
\frac{1}{\ell^{1/d}-1}\le \frac{d}{\log \ell}.
$$
Then,
\begin{align*}
r_\ell(L)&\le d^2L\ell^{-1}+2dL\ell^{-1}(\log \ell)^{-1}+dK_\ell(L)\\
&<(d+3)dL\ell^{-1}+dK_\ell(L).
\end{align*}
Therefore, recalling~\eqref{eq:bound K}, we obtain
$$
\ell^{r_\ell(L)}\le (2HL^d)^d\exp\((d+3)dL\frac{\log \ell}{\ell}\).
$$

Let  $\cL$ be the set of distinct prime divisors of $W(L)$. Then, we have
$$
|W(L)|\le (2HL^d)^{dt}\exp\((d+3)dL
\sum\limits_{\ell\in \cL}\frac{\log \ell}{\ell}\).
$$
Notice that
$$
\sum\limits_{\ell\in \cL}\frac{\log \ell}{\ell} = O\(\log t\),
$$
 because it is bounded by the sum over the first $t$ primes.
Hence,
\begin{equation}\label{prelim}
|W(L)|\le (2HL^d)^{dt}\exp(O(d^2L\log t)).
\end{equation}

Denoting by $T_1$ and $T_2$ the two terms in the product
on the right hand side of~\eqref{prelim} (so that $|W(L)| \le T_1T_2$),
we see that at least one of the inequalities
$|W(L)| \le T_1^2$ or $|W(L)| \le T_2^2$ holds.
More precisely, we have
\begin{equation}\label{upper1}
|W(L)|\le (2HL^d)^{2dt},
\end{equation}
or
\begin{equation}\label{upper2}
|W(L)| = \exp(O(d^2L\log t)).
\end{equation}

On the other hand, by Lemma~\ref{product}, if $L\ge 51(2d+1)$, we have
\begin{equation}\label{lower}
|W(L)|\ge (L/5)^{dL/18}.
\end{equation}

If~\eqref{upper1} holds, then
comparing~\eqref{upper1} and~\eqref{lower}, we find that
$$
t\ge \frac{c_1L}{d \log^{+} H},
$$
where $c_1$ is some absolute constant.

Alternatively, if~\eqref{upper2} holds, then
applying the same argument, but using~\eqref{upper2} and~\eqref{lower}, we obtain
$$
t\ge L^{c_2/d},
$$
where $c_2$ is an absolute constant. This completes the proof.
\end{proof}

\subsection{Proofs}\label{proofs}

\begin{proof}[Proof of Theorem~\ref{generic}]
If $\gcd(a_0,M)=1$ and $|a_0|>1$, then we pick a prime divisor $p$ of $a_0$. Then, $0$ a simple root of $f$ modulo $p$, and, clearly, $p\le H$.

Suppose $|a_0|=1$. Compute $f(\pm iM)$, $0\le i\le d$. Then, there exists at least one $i_0$ such that $|f(i_0M)|\ne 1$ or $|f(-i_0M)|\ne 1$. Assume $|f(i_0M)|\ne 1$ without loss of generality. Pick a prime divisor $p$ of $f(i_0M)$. Since $p\nmid M$, $i_0M$ is exactly a simple root of $f$ modulo $p$. So, $p\le 2H(dM)^d$.

Finally, suppose that $m=\gcd(a_0,M)>1$. Compute $f(\pm ia_0M)$, $0\le i\le d$. Then, there exists at least one $i_0$ such that $|f(i_0a_0M)|\ne |a_0|$ or $|f(-i_0a_0M)|\ne |a_0|$. Assume $|f(i_0a_0M)|\ne |a_0|$ without loss of generality. Pick a prime divisor $p$ of $f(i_0a_0M)/a_0$. Since $p\nmid M$, $i_0a_0M$ is exactly a simple root of $f$ modulo $p$. So, $p\le 2H(dHM)^d$.
\end{proof}

\begin{proof}[Proof of Theorem~\ref{polynomial}]
First, we note that for the irreducible polynomial $f$ we consider, 
since the units of $\Z[X]$ are exactly $\pm 1$, the content of $f$  is 1.

Let  $s = \omega(Q)$.
Clearly,
$$
s\le \frac{\log Q}{\log 2}< 2\log Q.
$$

Let $W(L)$ be the product of Lemma~\ref{lem:div prod}
and let $t=\omega(W(L))$ be the number of distinct prime divisors of $W(L)$.

Our goal is to show that for some sufficiently small $L$ we have
\begin{equation}\label{omegas}
s < t,
\end{equation}
which in turn immediately yields the bound
\begin{equation}\label{bound p}
p\le \max\{|f(j)|~:~ j=1, \ldots, L\}
\end{equation}
on the desired prime $p$.

By Lemma~\ref{lem:div prod} we either have
\begin{equation}\label{omega1}
t\ge \frac{c_1L}{d \log^{+} H},
\end{equation}
or 
\begin{equation}\label{omega2}
t\ge L^{c_2/d},
\end{equation}

If~\eqref{omega1} holds then 
it is sufficient to require that the inequality
\begin{equation}\label{L1}
L\ge c_3d  \log Q\log^{+} H
\end{equation}
for some absolute   constant $c_3>0$.

If~\eqref{omega2} holds then 
it suffices to require that
 \begin{equation}\label{L2}
L\ge (\log Q)^{c_4d}
\end{equation}
for some absolute   constant $c_4$.

Finally, comparing~\eqref{L1} with~\eqref{L2}, we choose
$$
L= \rf{C_0d \log Q\log^{+} H  + (\log Q)^{c_0d}},
$$
where $c_0$ and $C_0$ are some sufficiently
large absolute constants.
Now, from~\eqref{bound p} it is easy to see
that we can choose a prime
$$
p \le 2HL^d \le   C^dH\(d\log Q\log^{+} H \)^d+
H (\log Q)^{c d^2}
$$
for some absolute constants $c$ and $C$,
such that $f$ has a root modulo $p$ and $p\nmid Q$.
\end{proof}

\begin{proof}[Proof of Corollary~\ref{single root}] We recall
that, by \cite[Theorem~1]{Mahler1964} and~\eqref{trecia},
the discriminant $\Delta$ of $f$ satisfies
\begin{equation}\label{discriminant}
|\Delta|<d^{2d}H^{2d-2}.
\end{equation}
The required result now follows from Theorem~\ref{polynomial}.
\end{proof}

\section{Explicit form of Cassels' $p$-adic embedding theorem}\label{embedding0}

\subsection{Arbitrary number fields}\label{arbitrary}

Let $K$ be a number field of degree $d\ge 2$, and let $\beta_1,\ldots,\beta_n$ be some fixed non-zero elements of $K$. By Theorem~\ref{Cassels}, there exist infinitely many primes $p$ such that there is an embedding
\begin{equation}\label{embedding1}
\sigma: K\hookrightarrow \Q_p
\end{equation}
for which
$$
|\sigma(\beta_i)|_p=1, \quad \textrm{for $1\le i\le n$}.
$$
In order to prove Theorem~\ref{Cassels1}, we
 derive an upper bound for such a prime $p$.

First, we assume that $K=\Q(\alpha)$, and that the minimal polynomial of $\alpha$ over $\Z$ is $f$.
Put $$S=\{\beta_1,\ldots,\beta_n, \beta_{n+1}, \ldots,\beta_{2n}\},$$ where $\beta_{n+i}=\beta_i^{-1}$ for $1\le i\le n$. So, in order to ensure that $|\sigma(\beta_i)|_p=1$ for $1\le i\le n$, we only need to ensure that $|\sigma(\beta_i)|_p\le 1$ for $1\le i\le 2n$.

Note that every $\beta_i$, $1\le i\le 2n$, can be expressed uniquely by
$$
\beta_i=\frac{1}{b_i}(a_{i,0}+a_{i,1}\alpha+\cdots+a_{i,d-1}\alpha^{d-1}),
$$
where $b_i,a_{i,0},\ldots,a_{i,d-1}\in \Z$, $b_i\ge 1$, and $\gcd(a_{i,0},\ldots,a_{i,d-1})=1$. Moreover, for $1\le i\le 2n$, applying Corollary~\ref{coefficient}, we find that
\begin{equation}\label{coefficient1}
\log b_i
\le \max_{0\le j\le d-1}\h(a_{i,j}/b_i)
< d\h(\beta_i)+3d^2\h(\alpha)+ 2d^2.
\end{equation}

We claim that a prime $p$ satisfies~\eqref{embedding1} if it satisfies the following three conditions:
\begin{enumerate}
\item[{\bf A.}] $f(a)\equiv 0 \pmod p$ for some $a\in\Z$,

\item[{\bf B.}] $\Delta\not\equiv 0 \pmod p$, where $\Delta$ is
the discriminant of $f$,

\item[{\bf C.}] $b_i\not\equiv 0 \pmod p$, for $1\le i\le 2n$.
\end{enumerate}
Indeed, if $f$ satisfies Conditions~{\bf A} and~{\bf B}, then, by Hensel's lemma, there exists an element $\eta\in \Z_p$ such that $f(\eta)=0$, where $\Z_p$ denotes the set of $p$-adic integers. Then, we define an embedding $\sigma: K\to \Q_p$, by setting $\sigma(\alpha)=\eta$.
Under Condition~{\bf C}, we can see that $|\sigma(\beta_i)|_p\le 1$ for $1\le i\le 2n$.

Therefore, to get an upper bound for such smallest prime $p$ satisfying~\eqref{embedding1}, we can use Theorem \ref{polynomial} directly with $Q= 3|\Delta| b_1\cdots b_{2n}$, by applying \eqref{discriminant} and \eqref{coefficient1}. It follows that we can pick a prime $p$ satisfying \eqref{embedding1} and such that
$$
p\le H\left(dn\h(\alpha)+d\sum\limits_{i=1}^{n}\h(\beta_i)+d\log^{+} H+dn\right)^{O(d^2)},
$$
where $H=H(f)$ is the height of $f$.

In addition, by~\eqref{trecia}, we find that $H\le 2^d\exp(d\h(\alpha))$. So, we obtain
$$
p\le H\left(dn\h(\alpha)+d\sum\limits_{i=1}^{n}\h(\beta_i)+dn\right)^{O(d^2)},
$$
and
\begin{equation}\label{p1}
p\le \exp(d\h(\alpha))\left(dn\h(\alpha)+d\sum\limits_{i=1}^{n}\h(\beta_i)+dn\right)^{O(d^2)}.
\end{equation}

\begin{proof}[Proof of Theorem~\ref{Cassels1}] Since $K$ is generated by $\alpha_1,\ldots,\alpha_m\in K\setminus\Q$ over $\Q$, by Theorem~\ref{primitive}, there exists an algebraic number $\alpha$ such that $K=\Q(\alpha)$ and
$$
\h(\al)\leq  \log (dm)+\h(\alpha_1)+\cdots+\h(\alpha_m).
$$
Thus,
$$
\exp(d\h(\alpha)) \le (dm)^d \exp\(d\sum_{i=1}^m \h(\alpha_i)\)
$$
and
\begin{align*}
& dn\h(\alpha)+d\sum\limits_{i=1}^{n}\h(\beta_i)+dn \\
&\le d\(n \sum\limits_{i=1}^{m}\h(\al_i)+   \sum\limits_{i=1}^{n}\h(\beta_i)+n\log (dm)+n\) \\
& = O\(\(n \sum\limits_{i=1}^{m}\h(\al_i)+   \sum\limits_{i=1}^{n}\h(\beta_i)+n\log^{+} m\) d\log d\).
\end{align*}
Combining these two inequalities with~\eqref{p1}, we see that $p$
satisfies the inequality
\begin{align*}
p\le m^d\exp&\(d\sum\limits_{i=1}^{m}\h(\alpha_i)\)\\
& \left(n\sum\limits_{i=1}^{m}\h(\alpha_i)
+\sum\limits_{i=1}^{n}\h(\beta_i)
+n\log^{+}m\right)^{O(d^2)} d^{O(d^2)},
\end{align*}
which concludes the proof.
\end{proof}

\begin{proof}[Proof of Corollary~\ref{Cassels11}] It is easy to see that the result follows directly from
Theorem~\ref{Cassels1}.
\end{proof}

\begin{proof}[Proof of Corollary~\ref{Cassels12}] We only need to notice that for the fixed algebraic integers (respectively, units)  $\beta_1,\ldots,\beta_n\in\Z[\alpha]$, $b_i=1$ for $1\le i\le n$ (respectively, $1\le i\le 2n$). Then, the result follows directly from Corollary \ref{single root} and \eqref{trecia}.
\end{proof}

\begin{proof}[Proof of Corollary~\ref{Cassels13}] Since $K$ has at least one real embedding, by~\cite[Theorem~1.2]{Vaaler2013}, there exists an element $\alpha$ of $K$ such that $K=\Q(\alpha)$ and
$$
\h(\alpha)\le\frac{\log|D_K|}{2d}.
$$
Notice that $|D_K|\ge 7.25^d$ when $d\ge 16$, see~\cite[Section~2]{Odlyzko1975}.  Then, the desired result follows from~\eqref{p1}.
\end{proof}

\subsection{Cyclotomic fields}\label{cyclotomic0}

In this section, we consider the special case when $K$ is the $m$-th cyclotomic field with $m>2$, namely, $K=\Q(\zeta_m)$, where $\zeta_m$ is an $m$-th primitive root of unity.
Fix some non-zero elements $\beta_1,\ldots,\beta_n$ of $K$. We want to get an upper bound for the smallest prime $p$ such that there is an embedding
\begin{equation}\label{embedding2}
\sigma: K\hookrightarrow \Q_p
\end{equation}
for which
$$
|\sigma(\beta_i)|_p=1, \quad \textrm{for $1\le i\le n$}.
$$

In order to obtain a better bound, we need to refine~\eqref{omega1}
and~\eqref{omega2} in this special case. Here, we use the notation in Section~\ref{proofs} without special indication. We also note that in this case $f$ is the $m$-th cyclotomic polynomial, and the degree of $K$ (or $f$) is $d=\varphi(m)$.

\begin{proof}[Proof of Theorem~\ref{Cassels2}]

Recall that, for a prime $\ell$, we have $\ell^e \|m$. In particular, $e=0$ when $\ell\nmid m$. By the basic theory of cyclotomic fields (for example, see~\cite[Chapter 2]{Washtington1982}), $f$ has a root modulo $\ell$ if and only if $f$ can be factored completely modulo $\ell$, and if and only if $\ell\equiv 1 \pmod {m/\ell^{e}}$. In particular, if
$\ell\equiv 1 \pmod {m/\ell^{e}}$, then $f$ has $\varphi(m/\ell^{e})$ distinct roots modulo $\ell$. Moreover, if $\ell\mid m$, then $\ell\equiv 1 \pmod {m/\ell^{e}}$ is possible only when $\ell= P(m)$, where, as before,
$P(m)$ denotes the largest prime divisor of $m$.

Combining the above considerations with~\cite[Corollary~2]{Stewart1991}, for a prime $\ell\nmid m$ and any integer $k\ge 1$, we have
$$
N(\ell^{k}) \le \left\{ \begin{array}{ll}
                d,  & \textrm{if $\ell\equiv 1 \pmod m$},\\
               0, & \textrm{otherwise}.
                 \end{array} \right.
$$
So, for a prime $\ell\nmid m$, we obtain
\begin{equation}\label{r1}
r_\ell(L)\le L\sum\limits_{k=1}^{\infty}\frac{d}{\ell^k}+dK_\ell(L) \le 2dL\ell^{-1}+dK_\ell(L).
\end{equation}

Next, for any prime number $\ell$ and integer $k\ge 1$, it is easy to see that $N(\ell^k)\le \ell N(\ell^{k-1})\le\cdots\le \ell^{k-1}N(\ell)$. Then,
for a prime $\ell \mid m$ and $\ell^e \|m$, we find that
%$$
%N(\ell)=\left\{ \begin{array}{ll}
%                \varphi(m/\ell^e)  &
%\substack{\textrm{if $\ell$ is the largest prime divisor of $m$}\\
%\textrm{and $\ell\equiv 1 \pmod{m/\ell^{e}}$}},\\
%                \\
%                0 & \textrm{otherwise};
%                 \end{array} \right.
%$$
$$
N(\ell)=\left\{ \begin{array}{ll}
                \varphi(m/\ell^e),  & \textrm{if $\ell = P(m)$ and $\ell\equiv 1 \pmod{m/\ell^{e}}$},\\
                \\
                0, & \textrm{otherwise};
                 \end{array} \right.
$$
and for $k\ge 2$,
$$
N(\ell^k) \le \left\{ \begin{array}{ll}
                \varphi(m/\ell^e)\ell^{k-1},  & \textrm{if $\ell=P(m)$ and $\ell\equiv 1 \pmod{m/\ell^{e}}$},\\
                \\
                0, & \textrm{otherwise}.
                 \end{array} \right.
$$

Thus, for a prime $\ell\mid m$ and $\ell^e \|m$, applying the same arguments as those in Section~\ref{proofs}, we derive that
$$
r_\ell(L)\le\left\{ \begin{array}{ll}
                (\varphi(m/\ell^e)+3)dL\ell^{-1}+dK_\ell(L),  &
                \textrm{if $\ell= P(m)$}
                \\& \textrm{and $\ell\equiv 1 \pmod{m/\ell^{e}}$},\\
                \\
                dK_\ell(L), & \textrm{otherwise}.
                 \end{array} \right.
$$
Therefore, comparing this inequality with \eqref{r1}, for any prime $\ell$, we deduce
$$
r_\ell(L)\le 4dL\delta(m)\ell^{-1}+dK_\ell(L),
$$
where $\delta(m)$ has been defined in Section~\ref{introduction}.

Then applying the same arguments as Section~\ref{proofs}, for $L\ge 51(2d+1)$, we can deduce the following analogue of ~\eqref{omega1} and~\eqref{omega2}
$$
t\ge \frac{c_1L}{d\log^{+} H}\quad \textrm{or} \quad t\ge L^{c_2/\delta(m)},
$$
where $c_1$ and $c_2$ are two absolute constants, and $H=H(f)$.
So, for any integer $Q\ge 3$, we can choose a prime $p$ satisfying
$$
p \le C^dH\(d\log Q\log^{+} H \)^d+H(\log Q)^{c d\delta(m)}
$$
for some absolute constants $c$ and $C$, and
such that $f$ has a root modulo $p$ and $p\nmid Q$.

Finally, applying the same arguments as Section~\ref{arbitrary} and noticing that $\h(\zeta_m)=0$, we get the following upper bound for the smallest such prime number $p$ satisfying~\eqref{embedding2}
$$
p\le \left(d\sum\limits_{i=1}^{n}\h(\beta_i)+dn\right)^{O(d\delta(m))},
$$
where $d=\varphi(m)$.
\end{proof}

\section{Comments}

It is certainly interesting to understand how tight our
bounds are. Denoting by $p_k$ the $k$-th prime number and
defining
$$
\beta_i = \prod_{r=0}^{R-1} p_{nr+i}, \qquad i =1, \ldots, n,
$$
for some sufficiently large integer parameter $R$, we see from the
prime number theorem that
$$
\prod\limits_{i=1}^{n} \beta_i = \exp\((1+o(1)) nR \log (nR)\).
$$
On the other hand, the smallest prime $p$ with
$$
|\sigma(\beta_i)|_p=1,\quad \textrm{for $1\le i\le n$},
$$
obviously satisfies
$$
p > p_{nR} = (1+o(1)) nR \log (nR) =
(1+o(1)) \sum\limits_{i=1}^{n}\h(\beta_i).
$$

Here is a less obvious example, that illustrates the
sharpness of our results in Section~\ref{primitive0} for $d=2$. Although in our application
we do not need so strong result,
by a recent groundbreaking results of Maynard~\cite{Maynard}  and
Zhang~\cite{Zhang}, there
exists a positive integer $t$ such that $k+t$ and $k-t$ are both prime
for infinitely many positive integers $k$.
Take $k$ large enough and consider the following quadratic polynomial
$f_k(x)=x^2-2kx+t^2$ with height $2k$.
Its splitting field is $K=\Q(\sqrt{(k+t)(k-t)})$,
so each $\al$ satisfying $K=\Q(\al)$ is of the form $$\al=a+b\be \quad \text{with} \quad
\be=\sqrt{(k+t)(k-t)}$$ and rational $a$ and $b \ne 0$. We claim that $H(\al)>n/3$ for all such $\al$.
To prove this,
assume that $a_2x^2+a_1x+a_0 \in \Z[x]$, where $a_2>0$, is the minimal polynomial of $\al=a+b\be$
and write $b=b_0/b_1$ with coprime $b_0 \in \Z \setminus \{0\}$ and $b_1 \in \N$.
Note that the discriminant of $a_2x^2+a_1x+a_0$ is
\begin{align*}
a_1^2-4a_0a_2 &=a_2^2(a+b \be-a+b\be)^2\\
&=a_2^2(2b\be)^2=\frac{4a_2^2b_0^2(k+t)(k-t)}{b_1^2}.
\end{align*}
In particular, this yields that $b_1^2\mid 4a_2^2(k+t)(k-t)$.

Now, if  $k+t$ or $k-t$ is a prime divisor of $b_1$, then
this divisor also divides $a_2$. Thus, $H(\al) \geq |a_2| \ge k-t>k/2$, which is stronger than claimed.
If, otherwise, neither $k+t$ nor $k-t$ divides $b_1$, then $4a_2^2/b_1^2$ is an integer, so
$b_1^2 \le 4a_2^2 b_0^2$. It follows that
\begin{align*}
(k+t)(k-t) & = \frac{b_1^2}{4 a_2^2 b_0^2} (a_1^2-4a_0a_2) \le a_1^2 +4 |a_0| a_2
\\& \le 5\max\{|a_0|,|a_1|,|a_2|\}^2 =5H(\al)^2.
\end{align*}
This implies the inequality $H(\al)>k/3$, as claimed (provided that
$k$ is large enough).
Hence, our example shows that the exponent $(d-1)D/d$
in Corollary~\ref{pag} is sharp
for $d=2$ (in this case we automatically have $D=2$).

\section*{Acknowledgements}

The research of A.~D. was supported by the Research Council
of Lithuania Grant MIP-068/2013/LSS-110000-740.
The research of M.~S. and I.~E.~S. was supported by Australian
Research Council Grant DP130100237.

%We want to thank the referee for careful reading and very useful comments.

\end{document}